\documentclass[12pt]{article}
\usepackage{amsmath,amsthm, amssymb}
\usepackage{amssymb,latexsym}
\usepackage[hidelinks]{hyperref}
\usepackage{caption}

\newtheorem{theorem}{Theorem}

\newtheorem{lemma}{Lemma}

\newcommand{\seqnum}[1]{\href{https://oeis.org/#1}{\rm \underline{#1}}}

\begin{document}

\baselineskip=17pt

\title{\bf On $\boldsymbol{\psi}$-amicable numbers and their generalizations}

\author{\bf S. I. Dimitrov}

\date{2025}

\maketitle
\begin{abstract}
In this article, we study the properties of $\psi$-amicable numbers. We prove that their asymptotic density relative to the positive integers is zero. 
We also propose generalizations of $\psi$-amicable numbers.\\
\quad\\
\textbf{Keywords}:  Dedekind $\psi(n)$ function,  $\psi$-amicable numbers.\\
\quad\\
{\bf  2020 Math.\ Subject Classification}: 11A25 $\cdot$ 11D72 
\end{abstract}

\section{Notations, definitions and formulas}
\indent

The letter $p$, with or without a subscript, will always denote prime number.
Let $n>1$ be positive integer with prime factorization
\begin{equation*}
n=p^{a_1}_1\cdots p^{a_r}_r\,.
\end{equation*}
We define the Dedekind function $\psi(n)$ by the formula
\begin{equation}\label{psiform1}
\psi(n)=n \prod\limits_{i=1}^{r}\left(1+\frac{1}{p_i}\right) \quad \mbox{ and } \quad \psi(1)=1\,.
\end{equation}
Recall that
\begin{equation}\label{psiform2}
\psi(n)=\sum _{{d\,|\,n}}\frac{n\mu^2(d)}{d}\,,
\end{equation}
where $\mu(n)$ is the M\"{o}bius function. The proof of \eqref{psiform2} can be found in \cite{Penney}.
We shall use the convention that a congruence, $m\equiv n\,\pmod {d}$ will be written as $m\equiv n\,(d)$. 
A positive integer $n$ is said to be $\psi$-abundant if $\psi(n)>2n$.
A primitive $\psi$-abundant number is defined as an $\psi$-abundant number none of whose proper divisors is $\psi$-abundant.
Thus every $\psi$-abundant number is a multiple of a primitive $\psi$-abundant numbers.
Throughout this paper we denote $\nu=\log\log n$ and $s_\psi(n)=\psi(n)-n$. 

\section{Introduction and statement of the results}
\indent

Two different natural numbers $a$ and $b$ are said to be $\psi$-amicable if
\begin{equation}\label{psi-amicable}
\psi(a)=\psi(b)=a+b\,.
\end{equation}
In 2019, Amiram Eldar contributed sequences \seqnum{A323329} and \seqnum{A323330} to the OEIS \cite{Sloane}, listing the smaller and larger members, respectively, of the $\psi$-amicable pairs.
The smallest $\psi$-amicable pair is (1330, 1550).
Apparently, this definition is analogous to the classical definition of amicable pairs, which uses the sum-of-divisors function $\sigma$.
In Section \ref{ktuples}, we introduce the notion of $\psi$-amicable $k$-tuples.
In Section \ref{Aktuples}, we provide another definition of the same concept.
Our main result concerns the density of $\psi$-amicable pairs. We prove that their asymptotic density is zero.
\begin{theorem}\label{Theorem}
Let $M(n)$ denote the number of $\psi$-amicable pairs $(a, b)$ with $a<b$ and $a\leq n$. Then $M(n)=o(n)$ as $n\rightarrow \infty$.
\end{theorem}
Our approach is based on the method of Erdős' \cite{Erdos1955}. We essentially reproduce his argument, adapting it to Dedekind's $\psi$-function, with only minor technical modifications.

\section{Lemmas}
\indent

\begin{lemma}\label{Erdos1955}
Let $q_i$ be a sequence of prime numbers satisfying
\begin{equation*}
\sum_{i=1}^{\infty} \frac{1}{q_i} = \infty\,.
\end{equation*}
Denote by $v_q(n)$ the number of $q_i$ dividing $n$. Then the density of integers $n$ with $v_q(n) < A$ is $0$ for every $A$.
\end{lemma}
\begin{proof}
See (\cite{Erdos1955}, Lemma 1).
\end{proof}

\begin{lemma}\label{Erdos1934}
The number of integers $m \leq n$ which do not satisfy all of the following three conditions:
\[
\begin{aligned}
(1)\; & \text{if } p^{a} \mid m \text{ and } a > 1, \text{ then } p^{a} < (\log n)^{10}\,; \\
(2)\; & \text{the number of distinct prime factors of } m \text{ is less than } 10\nu\,; \\
(3)\; & \text{the greatest prime factor of } m \text{ is greater than } n^{1/(20\nu)}\,;
\end{aligned}
\]
is $o\!\left(\frac{n}{\log^{2} n}\right)$.
\end{lemma}
\begin{proof}
See (\cite{Erdos1934}, Lemma 1).
\end{proof}

\begin{lemma}\label{psinot}
Let $A$ be any constant. Then the density of integers $n$ for which
\begin{equation*}
\psi(n) \not\equiv 0 \, \Bigg(\bigg(\prod_{p\leq A}p\bigg)^A\Bigg)
\end{equation*}
is $0$ for every $A$.
\end{lemma}
\begin{proof}
It suffices to show that the density of integers $n$ for which there exists a prime $p \le A$ such that $\psi(n) \not\equiv 0 \, \big(p^A\big)$ is $0$. 
Let $q_1, q_2, \dots$ be primes satisfying $q_i \equiv -1 \, (p)$. It is well known that
\begin{equation*}
\sum_{i=1}^{\infty} \frac{1}{q_i} = \infty.
\end{equation*}
Hence, by Lemma \ref{Erdos1955}, the density of integers divisible by fewer than $A$ of the $q_i$ is $0$.
If $n$ is divisible by at least $A$ of the $q_i$, then \eqref{psiform1} gives us $\psi(n) \equiv 0 \, \big(p^A\big)$. 
Therefore the density of the integers with $\psi(n) \not\equiv 0 \, \big(p^A\big)$ is $0$.
\end{proof}

\begin{lemma}\label{psiA}
Denote
\begin{equation}\label{psiAn}
\psi_A(n) = \sum_{\substack{d \,|\, n \\ d\leq A}} \frac{n\mu^2(d)}{d}\,.
\end{equation}
Then for every $\varepsilon>0$ and $\eta>0$, there exists $A_0$ such that for $A>A_0$, the number of integers $n<x$ for which $\psi(n)-\psi_A(n)>\eta n$ is less than $\varepsilon x$.
\end{lemma}
\begin{proof}
Using \eqref{psiform2} and \eqref{psiAn}, we  have
\begin{equation}\label{sumpsipsiA1}
\sum_{n=1}^x \big(\psi(n)-\psi_A(n)\big)=\sum_{n=1}^x\sum_{\substack{d \,|\, n \\ d>A}} \frac{n\mu^2(d)}{d} 
=\sum_{d_1>A}\mu^2(d_1)\sum_{d_2\leq x/d_1}d_2<\sum_{d>A}\frac{x^2}{d^2}<\frac{x^2}{A}\,.
\end{equation}
If Lemma \ref{psiA} were not true, we would have $\psi(n)-\psi_A(n)>\eta n$ for at least $\varepsilon x$ integers $d\le x$. Thus
\begin{equation}\label{sumpsipsiA2}
\sum_{n=1}^x \big(\psi(n)-\psi_A(n)\big)>\eta\sum_{d\leq \varepsilon x}d>\frac{\eta \varepsilon^2 x^2}{4}\,.
\end{equation}
For $A >\frac{4}{\eta\varepsilon^2}$ \eqref{sumpsipsiA2} contradicts \eqref{sumpsipsiA1}, which proves Lemma \ref{psiA}.
\end{proof}

\begin{lemma}\label{primefact}
A primitive $\psi$-abundant number not exceeding $n$, which satisfies the three conditions of Lemma \ref{Erdos1934}, necessarily has a prime divisor between
$(\log n)^{10}$ and $n^{1/(40\nu)}$, provided $n$ is sufficiently large.
\end{lemma}
\begin{proof}
Assume that $m=ab$ is such a primitive $\psi$-abundant number, where all prime factors of $a$ are less than $(\log n)^{10}$
and all prime factors of $b$ are greater than  $n^{1/(40\nu)}$.
We have
\begin{equation}\label{sigmamest1}
\frac{\psi(m)}{m} \ge 2
\end{equation}
and
\begin{equation}\label{sigmaaest1}
\frac{\psi(a)}{a}<2\,.
\end{equation}
Now \eqref{sigmaaest1} and Lemma \ref{Erdos1934} imply
\begin{equation}\label{sigmaaest2}
\frac{\psi(a)}{a}\leq2-\frac{1}{a}<2-\frac{1}{(\log n)^{100\nu}}
\end{equation}
On the other hand by \eqref{psiform2} and Lemma \ref{Erdos1934}, we obtain 
\begin{equation}\label{sigmabest1}
\frac{\psi(b)}{b}=\sum _{{d\,|\,b}}\frac{\mu^2(d)}{d}=\prod_{p\,|\,b}\left(1+\frac{1}{p}\right)<\left(1+\frac{1}{n^{1/(40\nu)}}\right)^{10\nu}<1+\frac{20\nu}{n^{1/(40\nu)}}\,,
\end{equation}
if $n$ is sufficiently large. Now \eqref{sigmaaest2} and \eqref{sigmabest1} yield 
\begin{equation*}
\frac{\psi(m)}{m}=\frac{\psi(a)}{a}\frac{\psi(b)}{b}<2
\end{equation*}
for sufficiently large $n$, which contradicts \eqref{sigmamest1}.
\end{proof}

\section{Proof of Theorem \ref{Theorem}}  
\indent

Denote by $(a_i , b_i)$, $a_i<b_i$, $i = 1, 2, \ldots$ the sequence of pairs of $\psi$-amicable numbers. It is sufficient to prove that the sequence $a_i$, $i = 1, 2, \ldots$ has density 0. 
We split the sequence $a_i$ into two classes. Let $A=A(\varepsilon)$ be sufficiently large. In the first class are the $a_i$ for which there exists a $p\leq A$ with $\psi(a_i) \not\equiv 0 \, \big(p^A\big)$. 
It follows from Lemma \ref{psinot} that the density of the $a_i$ of the first class is 0. For the $a_i$ of the second class $\psi(a_i) \equiv 0 \, \big(p^A\big)$ for every $p\leq A$.
It is easy to see that if $d\leq A$ and $d \,|\, a_i$ then $\psi(a_i)-a_i \equiv 0\, (d)$. Therefore $\psi(a_i)-a_i=b_i \equiv 0 \, (d)$.
From Lemma \ref{psiA} it follows that except for at most $\varepsilon n$ of the $a_i$ not exceeding $n$ we have 
\begin{equation}\label{psiApsieta}
\frac{\psi_A(a_i)}{a_i}\geq\frac{\psi(a_i)}{a_i}-\eta\,.
\end{equation}
By \eqref{psiform2}, \eqref{psiAn}, \eqref{psiApsieta} and the fact that every divisor $d\leq A$ of $a_i$ also divides $b_i$, we get
\begin{equation}\label{psibpsia}
\frac{\psi(b_i)}{b_i}=\sum _{{d\,|\,b_i}}\frac{\mu^2(d)}{d}\geq\sum_{\substack{d \,|\, a_i \\ d\leq A}} \frac{\mu^2(d)}{d}=\frac{\psi_A(a_i)}{a_i}\geq\frac{\psi(a_i)}{a_i}-\eta\,.
\end{equation}
Now \eqref{psi-amicable} and \eqref{psibpsia} lead to
\begin{equation*}
\eta\geq\frac{\psi(a_i)}{a_i}-\frac{\psi(b_i)}{b_i}=\frac{b_i}{a_i}-\frac{a_i}{b_i}\,.
\end{equation*}
Hence
\begin{equation*}
1<\frac{b_i}{a_i}<1+\eta\,.
\end{equation*}
The last inequality and \eqref{psi-amicable} give us
\begin{equation}\label{psia2eta}
2<\frac{\psi(a_i)}{a_i}<2+\eta\,.
\end{equation}
Bearing in mind Lemma \ref{Erdos1934}, we may assume that each $a_i$ from \eqref{psia2eta} has a primitive $\psi$-abundant divisor satisfying all of the three conditions of Lemma \ref{Erdos1934}.
Let $a_1, a_2, \ldots, a_k$ denote all distinct numbers from \eqref{psia2eta} such that $a_i\leq n$.
According to Lemma \ref{primefact}, each $a_i$ has a prime factor $p_i$ between $(\log n)^{10}$ and $n^{1/(40\nu)}$. Thus $a_i=p_ic_i$, where $c_i<n/(\log n)^{10}$.
Suppose that $c_i=c_j$ for some $i\neq j$. Then $p_i\neq p_j$. We have
\begin{equation*}
\frac{\psi(a_i)}{a_i}=\frac{\psi(p_i)\psi(c_i)}{a_i}=\frac{\psi(c_i)}{c_i}\frac{p_i+1}{p_i}
\end{equation*}
and
\begin{equation*}
\frac{\psi(a_j)}{a_j}=\frac{\psi(p_j)\psi(c_j)}{a_j}=\frac{\psi(c_j)}{c_j}\frac{p_j+1}{p_j}\,,
\end{equation*}
which together imply
\begin{equation}\label{psiij1}
\frac{\psi(a_i)}{a_i}\frac{a_j}{\psi(a_j)}=\frac{p_j(p_i+1)}{p_i(p_j+1)}\,.
\end{equation}
Without loss of generality, assume that $p_j>p_i$. Now \eqref{psiij1} yields  
\begin{equation}\label{psiij2}
\frac{\psi(a_i)}{a_i}\frac{a_j}{\psi(a_j)}>1+\frac{1}{p_i(p_j+1)}>\frac{1}{n^{1/(20\nu)}}
\end{equation}
On the other hand, from \eqref{psia2eta} it follows that
\begin{equation*}
\frac{\psi(a_i)}{a_i}\frac{a_j}{\psi(a_j)}<1+\frac{\eta}{2}\,,
\end{equation*}
which contradicts \eqref{psiij2} for $\eta$ sufficiently small.
Consequently, $c_i\neq c_j$ for $i\neq j$, which means that the number of $a_i\leq n$ from \eqref{psia2eta} is equal to the number of $c_i$, which is $o(n)$.
This completes the proof of Theorem \ref{Theorem}.

\section{{\boldmath$\psi$}-amicable {\boldmath$k$}-tuples}\label{ktuples}  
\indent

Dickson \cite{Dickson} and Mason \cite{Mason} introduced a definition of amicable $k$-tuples using the sum-of-divisors function $\sigma$. 
We now provide an analogous definition based on the function $\psi$. We say that the natural numbers $n_1,\ldots, n_k$ form an $\psi$-amicable $k$-tuple if
\begin{equation*}
\psi(n_1)=\psi(n_2)=\cdots=\psi(n_k)=n_1+n_2+\cdots+n_k\,.
\end{equation*}
When $n_1<n_2<\cdots<n_k$, we have that
\begin{equation*}
kn_1<\psi(n_j)<kn_k
\end{equation*}
for each $j\in[1, k]$, which means that $n_1$ is $k$-$\psi$-abundant. 
The next theorem will help us search for $\psi$-amicable $k$-tuples.
\begin{theorem}
Suppose the natural numbers $N_1,\ldots, N_k$ and $a$ satisfy 
\begin{equation*}
(a, N_1)=\cdots=(a, N_k)= 1 
\end{equation*}
and
\begin{equation*}
\frac{\psi(a)}{a}=\frac{N_1+\cdots+N_k}{\psi(N_1)}=\cdots=\frac{N_1+\cdots+N_k}{\psi(N_k)}\,.
\end{equation*}
Then $aN_1,\ldots, aN_k$ are an $\psi$-amicable $k$-tuple.
\end{theorem}

\begin{proof}
This follows  directly from the multiplicativity of $\psi$.
\end{proof}
Several $\psi$-amicable triples are listed in the table below.
\begin{center}
\begin{tabular}[t]{|p{34em}|}
\hline \hspace{54mm}$\psi$-amicable triples \\
\hline
$(79170, 80850, 81900)$, $(150150, 158340, 175350)$, $(158340, 161700, 163800)$, \\
$(237510, 242550, 245700)$, $(300300, 316680, 350700)$, $(316680, 323400, 327600)$, \\
$(395850, 404250, 409500)$, $(450450, 474810, 526260)$, $(450450, 475020, 526050)$, \\
$(468930, 483210, 499380)$, $(474810, 485940, 490770)$, $(475020, 485100, 491400)$,\\
$(554190, 565950, 573300)$, $(570570, 662340, 702450)$, $(600600, 633360, 701400)$, \\  
$(622440, 641550, 671370)$, $(633360, 646800, 655200)$, $(641550, 646800, 647010)$, \\ 
$(644280, 644280, 646800)$, $(696150, 696150, 784980)$, $(712530, 727650, 737100)$\\
\hline
\end{tabular}
\captionof{table}{}\label{Table1}
\end{center}
The OEIS \cite{Sloane} sequences \seqnum{A385852}, \seqnum{A386901} and \seqnum{A386933} consist of the first, second and third components of $\psi$-amicable triples, respectively.

\section{{Another definition of \boldmath$\psi$}-amicable {\boldmath$k$}-tuples}\label{Aktuples}  
\indent

The following definition is analogous to that given by Yanney \cite{Yanney}, formulated for $\sigma$-function.
We say that the natural numbers $n_1,\ldots, n_k$ form a $\psi$-amicable $k$-tuple if
\begin{equation*}
\psi(n_1)=\psi(n_2)=\cdots=\psi(n_k)=\frac{1}{k-1}(n_1+n_2+\cdots+n_k)\,.
\end{equation*}
When $k=3$, we have 
\begin{equation*}
\left|\begin{array}{ccc}
n_1=s_\psi(n_2)+s_\psi(n_3)\\
n_2=s_\psi(n_1)+s_\psi(n_3)\\
n_3=s_\psi(n_1)+s_\psi(n_2)
\end{array}\right.\,.
\end{equation*}
Several $\psi$-amicable triples are listed in the table below.
\begin{center}
\begin{tabular}[t]{|p{36em}|}
\hline \hspace{60mm}$\psi$-amicable triples \\
\hline
$(2, 2, 2)$, $(4, 4, 4)$, $(6, 9, 9)$, $(8, 8, 8)$, $(16, 16, 16)$, $(18, 27, 27)$, $(28, 33, 35)$, $(32, 32, 32)$, \\
$(64, 64, 64)$, $(54, 81, 81)$, $(70, 99, 119)$, $(105, 124, 155)$, $(128, 128, 128)$, $(110, 135, 187)$, \\
$(165, 176, 235)$, $(150, 275, 295)$, $(200, 225, 295)$, $(182, 245, 245)$, $(162, 243, 243)$, \\
$(256, 256, 256)$, $(238, 255, 371)$, $(240, 385, 527)$, $(280, 345, 527)$, $(310, 315, 527)$,\\
$(310, 345, 497)$, $(315, 320, 517)$, $(315, 320, 517)$, $(382, 385, 385)$, $(364, 441, 539)$,\\  
$(512, 512, 512)$, $(512, 512, 512)$, $(468, 715, 833)$, $(520, 663, 833)$, $(585, 598, 833)$,\\ 
 $(644, 705, 955)$, $(590, 675, 895)$, $(486, 729, 729)$, $(795, 862, 935)$, (800, 885, 1195)\\
\hline
\end{tabular}
\captionof{table}{}\label{Table2}
\end{center}
The OEIS \cite{Sloane} sequence \seqnum{A387291} consists of the first elements of $\psi$-amicable triples.
When $k=4$, we have 
\begin{equation*}
\left|\begin{array}{cccc}
n_1=s_\psi(n_2)+s_\psi(n_3)+s_\psi(n_4)\\
n_2=s_\psi(n_1)+s_\psi(n_3)+s_\psi(n_4)\\
n_3=s_\psi(n_1)+s_\psi(n_2)+s_\psi(n_4)\\
n_4=s_\psi(n_1)+s_\psi(n_2)+s_\psi(n_3)
\end{array}\right.\,.
\end{equation*}
Several $\psi$-amicable quadruples are listed in the table below.
\begin{center}
\begin{tabular}[t]{|p{37.5em}|}
\hline \hspace{56mm}$\psi$-amicable quadruples  \\
\hline
$(3, 3, 3, 3)$, $(4, 4, 5, 5)$, $(6, 8, 11, 11)$, $(8, 8, 9, 11)$, $(9, 9, 9, 9)$, $(12, 14, 23, 23)$, \\
$(30, 44, 71, 71)$, $(44, 46, 55, 71)$, $(45, 45, 55, 71)$, $(51, 55, 55, 55)$, $(68, 68, 81, 107)$, \\
$(81, 81, 81, 81)$, $(99, 99, 115, 119)$, $(75, 95, 95, 95)$, $(96, 128, 161, 191)$, $(105, 155, 155, 161)$, \\
$(112, 112, 161, 191)$, $(100, 116, 145, 179)$, $(114, 158, 209, 239)$, $(152, 152, 177, 239)$, \\
$(152, 158, 171, 239)$, $(171, 171, 175, 203)$, $(188, 188, 235, 253)$, $(164, 166, 205, 221)$, \\
$(190, 236, 295, 359)$, $(225, 261, 275, 319)$, $(243, 243, 243, 243)$, $(186, 254, 329, 383)$, \\
 $(204, 230, 431, 431)$, $(230, 284, 391, 391)$, $(238, 272, 355, 431)$, $(255, 255, 355, 431)$\\
\hline
\end{tabular}
\captionof{table}{}\label{Table3}
\end{center}
The OEIS \cite{Sloane} sequence \seqnum{A387292} lists the first components of $\psi$-amicable quadruples.
When $k=5$, we have 
\begin{equation*}
\left|\begin{array}{cccc}
n_1=s_\psi(n_2)+s_\psi(n_3)+s_\psi(n_4)+s_\psi(n_5)\\
n_2=s_\psi(n_1)+s_\psi(n_3)+s_\psi(n_4)+s_\psi(n_5)\\
n_3=s_\psi(n_1)+s_\psi(n_2)+s_\psi(n_4)+s_\psi(n_5)\\
n_4=s_\psi(n_1)+s_\psi(n_2)+s_\psi(n_3)+s_\psi(n_5)\\
n_5=s_\psi(n_1)+s_\psi(n_2)+s_\psi(n_3)+s_\psi(n_4)
\end{array}\right.\,.
\end{equation*}
Several $\psi$-amicable quintuples are listed in the table below.
\begin{center}
\begin{tabular}[t]{|p{35em}|}
\hline \hspace{52mm}$\psi$-amicable quintuples \\
\hline
$(4, 5, 5, 5, 5)$, $(6, 9, 11, 11, 11)$, $(8, 9, 9, 11, 11)$, $(12, 15, 23, 23, 23)$, \\
$(28, 35, 35, 47, 47)$, $(32, 33, 33, 47, 47)$, $(30, 45, 71, 71, 71)$, $(36, 55, 55, 71, 71)$,\\ 
$(40, 51, 55, 71, 71)$, $(44, 51, 51, 71, 71)$, $(45, 46, 55, 71, 71)$, $(78, 117, 143, 167, 167)$,\\ 
$(98, 117, 123, 167, 167)$, $(104, 117, 117, 167, 167)$, $(84, 141, 161, 191, 191)$, \\
$(112, 155, 155, 155, 191)$, $(124, 161, 161, 161, 161)$, $(158, 175, 209, 209, 209)$,\\
$(158, 177, 177, 209, 239)$, $(140, 253, 253, 253, 253)$, $(176, 235, 235, 253, 253)$, \\
$(174, 225, 323, 359, 359)$, $(174, 261, 323, 323, 359)$, $(200, 261, 261, 359, 359)$\\
\hline
\end{tabular}
\captionof{table}{}\label{Table4}
\end{center}
The OEIS \cite{Sloane} sequence \seqnum{A387486} consists of the first elements of $\psi$-amicable quintuples.
When $k=6$, we have 
\begin{equation*}
\left|\begin{array}{cccc}
n_1=s_\psi(n_2)+s_\psi(n_3)+s_\psi(n_4)+s_\psi(n_5)+s_\psi(n_6)\\
n_2=s_\psi(n_1)+s_\psi(n_3)+s_\psi(n_4)+s_\psi(n_5)+s_\psi(n_6)\\
n_3=s_\psi(n_1)+s_\psi(n_2)+s_\psi(n_4)+s_\psi(n_5)+s_\psi(n_6)\\
n_4=s_\psi(n_1)+s_\psi(n_2)+s_\psi(n_3)+s_\psi(n_5)+s_\psi(n_6)\\
n_5=s_\psi(n_1)+s_\psi(n_2)+s_\psi(n_3)+s_\psi(n_4)+s_\psi(n_6)\\
n_6=s_\psi(n_1)+s_\psi(n_2)+s_\psi(n_3)+s_\psi(n_4)+s_\psi(n_5)
\end{array}\right.\,.
\end{equation*}
Several $\psi$-amicable sextuples are listed in the table below.
\vspace{-5mm}
\begin{center}
\begin{tabular}[t]{|p{38em}|}
\hline \hspace{60mm}$\psi$-amicable sextuples \\
\hline
$(5, 5, 5, 5, 5, 5)$, $(8, 8, 11, 11, 11, 11)$, $(9, 9, 9, 11, 11, 11)$, $(12, 16, 23, 23, 23, 23)$\\
$(14, 14, 23, 23, 23, 23)$, $(24, 28, 47, 47, 47, 47)$, $(25, 25, 25, 25, 25, 25)$, $(32, 32, 35, 47, 47, 47)$,  \\
$(30, 46, 71, 71, 71, 71)$, $(33, 33, 33, 47, 47, 47)$, $(36, 40, 71, 71, 71, 71)$, $(45, 51, 51, 71, 71, 71)$, \\
$(46, 46, 55, 71, 71, 71)$, $(98, 98, 143, 167, 167, 167)$, $(117, 123, 123, 143, 167, 167)$, \\
$(84, 112, 191, 191, 191, 191)$, $(105, 141, 141, 191, 191, 191)$, $(128, 128, 161, 161, 191, 191)$, \\
$(141, 141, 141, 155, 191, 191)$, $(155, 161, 161, 161, 161, 161)$, $(152, 152, 209, 209, 239, 239)$, \\
$(152, 158, 203, 209, 239, 239)$, $(158, 158, 203, 203, 239, 239)$, $(171, 171, 171, 209, 239, 239)$, \\
\hline
\end{tabular}
\captionof{table}{}\label{Table5}
\end{center}
The OEIS \cite{Sloane} sequence \seqnum{A387636} lists the first components of $\psi$-amicable sextuples.

\vskip20pt
\footnotesize
\begin{flushleft}
S. I. Dimitrov\\
\quad\\
Faculty of Applied Mathematics and Informatics\\
Technical University of Sofia \\
Blvd. St.Kliment Ohridski 8 \\
Sofia 1000, Bulgaria\\
e-mail: sdimitrov@tu-sofia.bg\\
\end{flushleft}

\begin{flushleft}
Department of Bioinformatics and Mathematical Modelling\\
Institute of Biophysics and Biomedical Engineering\\
Bulgarian Academy of Sciences\\
Acad. G. Bonchev Str. Bl. 105, Sofia 1113, Bulgaria \\
e-mail: xyzstoyan@gmail.com\\
\end{flushleft}

\end{document}